\documentclass{amsart}
\usepackage{amsmath,amssymb,amsthm,hyperref,mathtools}
\usepackage{color,tikz}
\usepackage{cleveref}
 \usetikzlibrary{calc}

\newtheorem{thm}{Theorem}[section]
\newtheorem{cor}[thm]{Corollary}

\newtheorem{lem}[thm]{Lemma}
\newtheorem{prop}[thm]{Proposition}
\theoremstyle{remark}

\theoremstyle{definition}

\newtheorem{defn}[thm]{Definition}

\newcommand{\Cay}[2]{\mathrm{Cay}(#1,#2)}		
\newcommand{\etdom}{\gamma_{\mathrm{all}}^\infty}

\newcommand{\ov}{\mathrm{ov}}

\title{Eternal domination in Cayley graphs}
\author{MacKenzie Carr}
\address{Department of Mathematics\\ Toronto Metropolitan University \\ Toronto, ON M5B 2K3\\
Canada}
\email{mackenzie.carr@torontomu.ca}

\author{Nancy E.~Clarke}
\thanks{N.E.~Clarke acknowledges research support from NSERC (2020-06528).}

\address{Department of Mathematics and Statistics\\
Acadia University \\
Wolfville, NS B4P 2R6\\
Canada}
\email{nancy.clarke@acadiau.ca}

\author{Gary MacGillivray}
\thanks{Supported by the Natural Sciences and Engineering Research Council of Canada (2017-04459)}
\address{Department of Mathematics and Statistics\\
University of Victoria\\
Victoria, BC  V8W 2Y2\\
Canada}
\email{gmacgill@uvic.ca}

\author{Joy Morris}
\thanks{Supported by the Natural Science and Engineering Research Council of Canada (grant RGPIN-2024-04013).}

\address{Department of Mathematics and Computer Science\\
	University of Lethbridge\\
	Lethbridge, AB T1K 3M4\\
	Canada}

\email{joy.morris@uleth.ca}

\begin{document}

\begin{abstract}
    Eternal domination is a process in which a set of guards occupying a dominating set on a graph protects against an infinite sequence of attacks. After a vertex is attacked, one guard must move along an edge to the attacked vertex and each of the remaining guards may move along an edge so that the guards again occupy a dominating set on the graph and can defend the next attack. The minimum number of guards needed in a graph $\Gamma$ is the eternal domination number, denoted by $\etdom(\Gamma)$. In this paper, we show that the eternal domination number of a vertex-transitive graph with an efficient dominating set is equal to its domination number. We show that a Cayley graph on a generalized dihedral group whose connection set contains few or many reflections is efficiently dominated. Then, we provide an infinite family of connected Cayley graphs for which $\etdom(\Gamma) = \gamma(\Gamma)+1$, generalizing a result of [Braga et al., \emph{J.~Combin.~Math.~Combin.~Comput.~}96 (2016), 13--22]. Finally, we build an infinite family of connected Cayley graphs with $\etdom(\Gamma) \geq \gamma(\Gamma)+2$. 
\end{abstract}


\maketitle

\section{Introduction}
Imagine placing mobile sensors, or \textit{guards} on the vertices of a graph.  These guards then \textit{reconfigure} in response to events, or \textit{attacks} that occur at vertices.  In the reconfiguration step, some guard must slide along an edge to the attacked vertex and the remaining guards either stay in place or slide along an edge to an adjacent vertex in such a way that the guards can then \textit{defend} the next attack.
The parameter of interest is the minimum number of guards so that this is possible.  

We consider only the model where all guards can move in response to an attack, and the sequence of attacks is infinite.  There are other models in which only one guard can move or the sequence of attacks has limited length;   see \cite{KM16, W23, V24} for further details.
In addition, we allow only one guard to occupy a vertex.  There are examples in the literature showing that, in the all-guards move model fewer guards can be required otherwise \cite{FGMO18, W23}.

In this context, an \textit{eternal dominating configuration} for a graph $\Gamma$ is a collection $\mathcal{D} = \{D_1, D_2, \ldots, D_\ell\}$ of subsets of $V$ such that $|D_i| = |D_j|$, and for every set $D_i \in \mathcal{D}$ and vertex $w \not\in D_i$, there exists a set $D_j \in \mathcal{D}$ with $w \in D_j$ so that guards located at vertices in $D_i$ can reconfigure to be located at vertices in $D_j$.  
A set $D \subseteq V$ is an \textit{eternal dominating set} if it belongs to an eternal dominating configuration.  The smallest integer for which there exists an eternal dominating set is the \textit{eternal domination number}, and is denoted by $\etdom(\Gamma)$ or $\etdom$ if the context is clear.

Eternal domination can be viewed as a 2-player discrete-time game. One player controls the initial position and the movement of the guards and the other player selects which vertex to attack.  After the guards are located on vertices, the players alternate turns.  On each turn, the attacker chooses a vertex and then the opponent reconfigures the guards so that one of them occupies the attacked vertex.  The attacker wins if the guards are ever unable to respond.  Otherwise the opponent wins.

Let $G$ be a group, and
let $S \subseteq G$ such that  if $s \in S$ then $s^{-1} \in S$.
The \textit{Cayley graph with connection set $S$}, denoted $\Gamma = $\textit{Cay$(G, S)$}, has vertex set $V(\Gamma) = G$ and an edge from $u$ to $v = us$ for every $s \in S$.
Cayley graphs are vertex transitive.  If $xv = w$ then the function $\phi: V \to V$ defined by $\phi(v) = xv$ is an automorphism of $G$ that sends $v$ to $w$. 

The \textit{domination number} of a graph, denoted by $\gamma(\Gamma)$ or $\gamma$ if the context is clear, is the smallest size of a set of vertices $D$ such that every vertex not in $D$ is adjacent to a vertex of $D$.  Clearly the domination number is a lower bound for the eternal domination number.  One upper bound for the eternal domination number is $2\gamma$ and equality can occur for each value of $\gamma \geq 3$ \cite{KM16}.

Eternal domination in the all guards move model was introduced in \cite{GHH2005}.
It it true that Cayley graphs on abelian groups have equal eternal domination number and domination number.
By confusing right and left multiplication, it was further claimed in \cite{GHH2005} that $\etdom = \gamma$ for all Cayley graphs.
The first counterexamples were found by Braga, de Sousa and Lee \cite{BSL2016}.
They examined 7871 Cayley graphs and found that 61 of them had $\etdom = 1 + \gamma$ and the rest had $\etdom = \gamma$.  It was left as an open problem whether there exists a connected Cayley graph with $\etdom > 1 + \gamma$.

In the next section we show that if a vertex-transitive graph has an efficient dominating set, then it is an eternal dominating set.  Hence $\etdom = \gamma$ for these graphs.  In the subsequent three sections we explore infinite families of Cayley graphs for which the eternal domination number is equal to, one greater than, and two greater than the domination number.  The last section yields a positive partial solution to the question raised in \cite{BSL2016}.

\section{Efficient dominating sets}

In this section, we prove that in the case of vertex-transitive graphs, if there is an efficient dominating set then it is an eternal dominating set. This means that for Cayley graphs that have an efficient dominating set, the problem of eternal domination is solved.

Of course, in a regular graph with $v$ vertices and valency $k$, each vertex dominates $k+1$ vertices, so an efficient dominating set must consist of $v/(k+1)$ vertices. This means that an efficient dominating set can only exist if $k+1 \mid v$. So it will never be the case that this result gives us information about all Cayley graphs on some group (unless the group is extremely small). Indeed, the problem of finding whether an efficient dominating set exists in a Cayley graph even when this divisibility condition is satisfied is highly nontrivial. Efficient dominating sets are also known as perfect codes, and the problem has been studied using either terminology, but even for cyclic groups results are incomplete. See, for example, \cite{CM2013,DSLW2017,D2014,FHZ2017,FHKL2025,KM2013}.

\begin{prop}\label{prop:efficient}
    If every vertex of the graph $\Gamma$ lies in some efficient dominating set, then $\etdom(\Gamma)=\gamma(\Gamma)$.
\end{prop}

\begin{proof}
    We show that the collection $\mathcal D$ of all efficient dominating sets for $\Gamma$ is an eternal dominating configuration.
    
    Let $D\in \mathcal D$, and consider any vertex $v \notin D$. By assumption, $v \in D'$ for some $D' \in \mathcal D$.

    Since $D$ is an efficient dominating set, each vertex of $D'$ is dominated by a unique vertex of $D$. Since $D'$ is an efficient dominating set, each vertex of $D$ is dominated by a unique vertex of $D'$, so (since domination is a symmetric relationship) in particular no two vertices of $D'$ can be dominated by the same vertex of $D$.

    This implies that there is a way to reconfigure guards from $D$ to $D'$. 
\end{proof}

\begin{cor}
    If a vertex-transitive graph $\Gamma$ has an efficient dominating set, then $\etdom(\Gamma)=\gamma(\Gamma)$.
\end{cor}

\begin{proof}
    Since the graph is vertex-transitive, the existence of an efficient dominating set implies that every vertex lies in an efficient dominating set.~\Cref{prop:efficient} now completes the proof.
\end{proof}

\section{Cayley graphs on generalised dihedral groups}

Although the typical definition of a Cayley graph assumes that the identity element of the group does not lie in the connection set, this is done only to ensure that the graphs are simple. Including the identity element adds a loop at every vertex. For many purposes such a loop is irrelevant. In the context of this paper, including the identity in the connection set (and a loop at every vertex) allows us to simplify some of our notation and arguments, as it ensures that (with $S$ representing the connection set of the Cayley graphs) the vertices dominated by a given vertex $g$ are the elements of $gS$ rather than $gS \cup \{g\}$. Accordingly, we will typically include the identity in our connection sets.

Although they use their own terminology in their statement, the following result appears in~\cite{BSL2016}.

\begin{thm}[\cite{BSL2016}, Theorem 2]\label{like-abelian}
    If $\Gamma=\Cay{G}{S} \cong \Cay{A}{C}$ where $A$ is an abelian group, then $\etdom(\Gamma)=\gamma(\Gamma).$
\end{thm}

Their proof boils down to the ideas from the paper~\cite{GHH2005}.  Due to confusion over left-multiplication versus right-multiplication, the original proof from~\cite{GHH2005} is invalid without additional hypotheses. These are most easily fixed by simply assuming the group to be abelian. In the following theorem, we provide a direct proof of a generalisation of the central idea of~\cite{GHH2005}. Our main goal in this section is to use this to obtain a bit more information about the eternal domination number of some Cayley graphs on generalised dihedral groups, even when they are not isomorphic to some Cayley graph on an abelian group.

\begin{thm}\label{abelian}
    Let $D$ be a dominating set for the Cayley graph $\Cay{G}{S}$, and let $s \in S$ be such that $s$ is either centralised or inverted by every element of $G$. Then for any vertex $v$ in $Ds$ there is a dominating set $D'$ containing $v$ such that guards can be reconfigured from $D$ to $D'$.
\end{thm}

\begin{proof}
    Observe that since every element of $G$ either centralises or inverts $s$, we can partition $G$ into $G_1=C_G(s)$ and $G_2=G \setminus G_1$. Furthermore, either $G=G_1$ or $G_1$ has index $2$ in $G$. Similarly, we can partition $D$ into $D_1=D\cap G_1$ and $D_2=D\cap G_2$, and $S$ into $S_1=S\cap G_1$ and $S_2=S\cap G_2$.

    Let $v=ds$ where $d \in D_i$, with $i \in \{1,2\}$. Then for every $d' \in D_i$, we move the defender from $d'$ to $d's$, while for every $d' \in D_{3-i}$, we move the defender from $d'$ to $d's^{-1}$ (we can do this since $S$ is inverse-closed). We claim that the vertices now occupied by the defenders form a dominating set $D'$ for $\Cay{G}{S}$.

    Observe that $D$ being a dominating set for $\Cay{G}{S}$ is equivalent to saying $G = DS$. (Here we are using the assumption that $1 \in S$.) We can refine this using our partitions to the following equivalent statement: $G_1=D_1S_1 \cup D_2S_2$ and $G_2=D_1S_2 \cup D_2S_1$.

    Although the proofs are very similar, we analyze the situations $i=1$ and $i=2$ separately. 

    Suppose first $i=1$. Then after moving, the defenders occupy the vertices $D_1s \cup D_2s^{-1}$, and $$(D_1s \cup D_2 s^{-1})S  = D_1sS_1 \cup D_1sS_2 \cup D_2s^{-1}S_1 \cup D_2s^{-1}S_2 $$
    $$ = D_1S_1s \cup D_1S_2s^{-1} \cup D_2S_1s^{-1}\cup D_2S_2s,$$ since $s$ commutes with every element of $S_1$ and is inverted by every element of $S_2$. Since $G_1 = C_G(s)$ is a subgroup of $G$ that contains $s$, we have that $G_1 = G_1s$. Similarly, $s^{-1}\in G_1$ and $G_2$ is either empty or the unique nontrivial coset of $G_1$ in $G$, so $G_2s^{-1}=G_2$. Grouping the terms above, we have 
    $$(D_1s\cup D_2s^{-1})S = (D_1S_1\cup D_2S_2)s \cup (D_1S_2\cup D_2S_1)s^{-1}$$
    $$ = G_1s\cup G_2s^{-1} = G_1\cup G_2 = G.$$
    Hence the defenders do occupy a dominating set after moving.
    
Now suppose $i=2$. Then after moving, the defenders occupy the vertices $D_1s^{-1} \cup D_2s$. 
Then $$(D_1s^{-1} \cup D_2 s)S  = D_1s^{-1}S_1 \cup D_1s^{-1}S_2 \cup D_2sS_1 \cup D_2sS_2$$
    $$ = D_1S_1s^{-1} \cup D_1S_2s \cup D_2S_1s\cup D_2S_2s^{-1},$$ since $s$ commutes with every element of $S_1$ and is inverted by every element of $S_2$. Since $G_1 = C_G(s)$ is a subgroup of $G$ that contains $s^{-1}$, we have that $G_1 = G_1s^{-1}$. Similarly, $s\in G_1$ and $G_2$ is either empty or the unique nontrivial coset of $G_1$ in $G$, so $G_2s=G_2$. Grouping the terms above, we have 
    $$(D_1s^{-1}\cup D_2s)S = (D_1S_1\cup D_2S_2)s^{-1} \cup (D_1S_2\cup D_2S_1)s $$ $$= G_1s^{-1}\cup G_2s = G_1\cup G_2 = G.$$
    Hence the defenders do occupy a dominating set after moving.
   \end{proof}

For clarity, we now define generalised dihedral groups.

\begin{defn}
    Let $A$ be an abelian group. The generalised dihedral group over $A$ is the group $$\langle A, x: x^2=1, x^{-1}ax=a^{-1} \ \forall a \in A\rangle.$$
\end{defn}

The following result is the best result we know of, to determine a fairly large class of Cayley graphs on generalised dihedral groups that are isomorphic to Cayley graphs on abelian groups. For these graphs, Theorem~\ref{like-abelian} applies directly to yield the eternal domination number in terms of the domination number.

\begin{thm}[\cite{MS2021}, Theorem 2.3]\label{isomorphic}
    Let $G$ be a generalised dihedral group over the abelian group $A$, and $\Gamma=\Cay{G}{S}$. If there is some $y \in xA$ such that for every $a \in A$ we have $ya \in S \cap xA$ if and only if $ya^{-1} \in S \cap xA$, then $\Gamma$ is isomorphic to a Cayley graph on the abelian group $A \times C_2$.
\end{thm}

It is not necessarily obvious or easy to determine whether or not a given connection set $S$ actually has the property defined in Theorem~\ref{isomorphic}. The following result highlights a few situations in which it is relatively easy to see that the condition is satisfied.

\begin{cor}\label{one-reflection}
    Let $G$ be a generalised dihedral group of order $2n$ over the abelian group $A$, and let $\Gamma=\Cay{G}{S}$. If $|S\cap xA| \le 1$ or $|S \cap xA| \ge n-1$ then $\Gamma$ is isomorphic to a Cayley graph on the abelian group $A \times C_2$, and therefore $\etdom(\Gamma)=\gamma(\Gamma)$.
\end{cor}

\begin{proof}
    If we can show that $S$ satisfies the condition of Theorem~\ref{isomorphic}, this implies that $\Gamma$ is isomorphic to a Cayley graph on an abelian group, and therefore by Theorem~\ref{like-abelian}, the result follows.

    If $|S \cap xA|=0$ then the condition of Theorem~\ref{isomorphic} is satisfied vacuously. Similarly, if $S \cap xA=xA$ then the condition is satisfied for any choice of $y \in xA$.

    If $|S\cap xA|=1$ then taking $y$ to be the unique element of $S \cap xA$  it is easy to see that the condition of Theorem~\ref{isomorphic} is satisfied. Similarly, if $|S \cap xA|=n-1$ then it is straightforward to verify that the condition of Theorem~\ref{isomorphic} is satisfied by taking $y$ to be the unique element of $xA$ that is not in $S$.
\end{proof}

We conclude this section with an argument that allows us to expand the conditions of Corollary~\ref{one-reflection} slightly (increasing the size of the intersection of the connection set with the nontrivial coset of $A$ by one). For even this slight expansion the argument we use is not at all trivial.

\begin{thm}
    Let $G$ be a generalised dihedral group of order $2n$ over the abelian group $A$, and let $\Gamma=\Cay{G}{S}$. If $|S \cap xA| \le 2$ then $\etdom(\Gamma)=\gamma(\Gamma)$.
\end{thm}

\begin{proof}
    If $|S \cap xA|\le 1$ then the result follows from~\Cref{one-reflection}, so we may assume $|S\cap xA|=2$. We will show that the set of all minimal dominating sets for $\Gamma$ is an efficient dominating configuration.

    Let $D$ be a minimal dominating set for $\Gamma$, and let $v=ds$ for some $d \in D$ and $s \in S \cap A$. Since $s \in A$, every element of $G$ either centralises or inverts $s$, so~\Cref{abelian} tells us that we can reconfigure the defenders to a minimal dominating set $D'$ that includes $v$.

    Let $s,s'$ be the two elements of $S \cap xA$. Since $D$ is a dominating set, the only remaining possibility is that for some $d \in D$, $v=ds$ or $v=ds'$. Observe that the edges that come from $s$ and $s'$ induce a 2-factor in $\Gamma$, consisting of a collection of cycles of length $2|ss'|$ that alternate between vertices in $A$ and vertices in $xA$. 

    Observe that the map $\varphi$ given by $\varphi(g)=gs$ for $g \in dA$ and $\varphi(g)=gs'$ for $g \in xdA$ is a bijection on $G$.
    For each $d' \in D$, we reconfigure the defender from $d'$ to $\varphi(d')$, which is clearly adjacent to $d'$. We will now show that for every $g \in G$, $\varphi(g)$ is dominated by some vertex in $\varphi(D)$; since $\varphi$ is a bijection on $G$, this implies that $\varphi(D)$ is a minimal dominating set.

    We consider a few cases, based on the relationship between $D$ and $g$. First suppose that $g \in D$. Then $\varphi(g)\in \varphi(D)$ is dominated by $\varphi(D)$.

    Next suppose that $g=d'a$ for some $d' \in D$ and some $a \in A \cap S$. Then either $d',g \in dA$ or $d',g \in xdA$, so either $\varphi(d')=d's$ and $\varphi(g)=gs$, or $\varphi(d')=d's'$ and $\varphi(g)=gs'$. Choose $s'' \in \{s,s'\}$ so that $\varphi(d')=d's''$ and $\varphi(g)=gs''$. Since $g=d'a$ and $s'' \in xA$ inverts $a$, we have $\varphi(g)=d'as''=d's''a^{-1}=\varphi(d')a^{-1}$. Since $S$ is inverse-closed, $a^{-1} \in A \cap S$, so $\varphi(g)$ is dominated by $\varphi(d') \in \varphi(D)$.

    Finally, suppose $g=d's''$ for some $d' \in D$ and some $s'' \in \{s,s'\}$. Then for some $s_1, s_2$ with $\{s_1,s_2\}=\{s, s'\}$, we have $\varphi(d')=d's_1$ and $\varphi(g)=gs_2$. If $s_1=s''$ then $\varphi(d')=g$ and since $\varphi(g)=gs_2=\varphi(d')s_2$ and $s_2\in S$, $\varphi(g)$ is dominated by $\varphi(d') \in \varphi(D)$. The other possibility is that $s_2=s''$. In this case, since $s''$ is an involution and $g=d's''$ we have $d'=gs''=\varphi(g)$. Since $\varphi(d')=d's_1$ is adjacent to $d'=\varphi(g)$, again $\varphi(g)$ is dominated by $\varphi(d') \in \varphi(D)$. 

    From these cases we conclude that $\varphi(D)$ is indeed a dominating set, and therefore we have an eternal dominating configuration.
\end{proof}

\section{Cayley graphs whose eternal domination number is greater than their domination number}

In~\cite{BSL2016}, the authors find a Cayley graph $\Gamma$ that they are able to prove has $\etdom(\Gamma)=\gamma(\Gamma)+1$. Their example is a Cayley graph on the group we refer to as $G_3$, below. As discussed in the introduction, they computationally found a number of graphs with this property, but were unable to build either an infinite family, or a Cayley graph $\Gamma$ with $\etdom(\Gamma)>\gamma(\Gamma)+1$, except by taking disjoint copies of their example to produce disconnected graphs.

In this section we generalise their work to produce an infinite family of connected Cayley graphs,  on a family of groups $G_k$ for each odd $k \ge 3$, all of which have the property that $\etdom(\Gamma)=\gamma(\Gamma)+1$. In the following section we will further build on these ideas to produce an infinite family of Cayley graphs such that for each graph $\Gamma$ in the family, $\etdom(\Gamma)>\gamma(\Gamma)+1$. 

Our methods grow significantly in complexity as we try to increase the gap between the domination and eternal domination numbers; in fact we did not believe there to be sufficient value to justify the work that would be required to prove that the gap is exactly $2$ in the examples we produce in the next section, although this seems to be the case. To guarantee higher gaps may well be feasible, but our ideas do not seem likely to make it reasonable to find Cayley graphs for which the gap is arbitrarily large, or to prove it if we did find them.

We begin by introducing the family of groups on which the Cayley graphs in this section will be defined.

\begin{defn}
    Let $k \ge 3$ be odd. We define $$G_k=\langle \rho, \tau,a: \rho^k=\tau^2=a^k=e, a \in Z(G_k), \tau\rho=\rho^{-1}\tau\rangle,$$ where $e$ denotes the identity and $Z(G_k)$ the center of $G_k$.
    Notice that $G_k \cong D_{2k} \times C_k$ has order $2k^2$.

    We also define the following subgroups of $G_k$: 
    \begin{align*}
  H_k&=\langle \rho,a\rangle \cong C_{k} \times C_k,\\ A_k&=\langle a \rangle \cong C_{k}, \\K_k&=\langle \rho a\rangle \cong C_k,\text{ and}\\ K'_k&=\langle \rho a^{-1}\rangle \cong C_k.
    \end{align*}
\end{defn}
Note that $H_k$ and $A_k$ are normal in $G_k$, while $K_k$ and $K'_k$ are not. However, $K_k$ and $K_k'$ are normal in $H_k$, so within that context we may occasionally refer to their ``cosets"; within $G_k$ we will be considering left cosets.

\begin{defn}
    Let $k \ge 3$. We use $\Gamma_k=\Cay{G_k}{S_k}$, where 
    $$S_k=\{\tau a^i,\rho^i a^i:1 \le i \le k-1\} \cup \{\rho\tau a,\rho\tau a^{-1}\}.$$
\end{defn}

In everything that follows, calculations are performed modulo $k$ wherever they relate to powers of $a$ or of $\rho$.

To assist the reader's intuition we provide a couple of diagrams (\Cref{fig1} and \Cref{fig2}) that show the neighbours of $e$ and of $\tau$ in $\Gamma_5$, and explain what the cosets of each of the subgroups above looks like in these diagrams. Note that because $G_k$ is nonabelian, we can either produce a diagram in which left-multiplication by $\rho$ appears to act ``the same" on all cosets of $A_k$, or one in which left-multiplication by $\tau$ does, but not both. We make the former choice, so in both cosets of $H_k$, left-multiplication by $\rho_k$ takes us down by one row, but left-multiplication by $\tau$ reverses the order of the rows in each copy while interchanging the two copies of $H_k$.

In these figures, the cosets of $H_5$ are the left and right halves of each diagram. Each coset of $A_5$ is a row in one half of a diagram. Left cosets of $K_5$ are forward diagonals in the left half of a diagram, and backward diagonals in the right half, while left cosets of $K_5'$ are backward diagonals in the left half of a diagram, and forward diagonals in the right half. In the next section we will also consider left cosets of $\langle A_5, \tau\rangle$; these are rows that run all the way across a diagram.

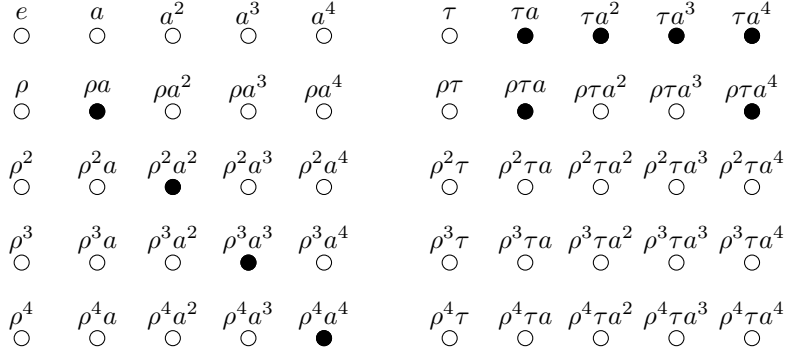
\begin{figure}
    \begin{tikzpicture}
        \foreach \x in {0,...,4}{
        \foreach \y in {0,...,4}{
\node[draw,circle, fill=white,inner sep=2pt] (\x) at (\x,\y) {};
}}
        \foreach \x in {2,...,4}{
        \foreach \y in {2,...,4}{
\node at ($(\x,4-\y)+(0,0.3)$) {$\rho^\y a^\x$};
}}
        \foreach \y in {2,...,4}{
\node at (\y,3.3) {$\rho a^\y$};
\node at (\y,4.3) {$a^\y$};
\node at ($(0,4-\y)+(0,0.3)$) {$\rho^\y$};
\node at ($(1,4-\y)+(0,0.3)$) {$\rho^\y a$};
}
\node at (0,3.3) {$\rho$};
\node at (0,4.3) {$e$};
\node at (1,4.3) {$a$};
\node at (1,3.3) {$\rho a$};
\foreach \x in {1,...,4}
{
\node[draw,circle, fill=black,inner sep=2pt] at (\x,4-\x) {};
}
    \end{tikzpicture}
    \qquad
      \begin{tikzpicture}
        \foreach \x in {0,...,4}{
        \foreach \y in {0,...,4}{
\node[draw,circle, fill=white,inner sep=2pt] (\x) at (\x,\y) {};
}}
        \foreach \x in {2,...,4}{
        \foreach \y in {2,...,4}{
\node at ($(\x,4-\y)+(0,0.3)$) {$\rho^\y \tau a^\x$};
}}
        \foreach \y in {2,...,4}{
\node at (\y,3.3) {$\rho \tau a^\y$};
\node at (\y,4.3) {$\tau a^\y$};
\node at ($(0,4-\y)+(0,0.3)$) {$\rho^\y \tau$};
\node at ($(1,4-\y)+(0,0.3)$) {$\rho^\y \tau a$};
}
\node at (0,3.3) {$\rho \tau$};
\node at (0,4.3) {$\tau$};
\node at (1,4.3) {$\tau a$};
\node at (1,3.3) {$\rho \tau a$};
\foreach \x in {1,...,4}
{
\node[draw,circle, fill=black,inner sep=2pt] at (\x,4) {};
}
\node[draw,circle, fill=black,inner sep=2pt] at (1,3) {};
\node[draw,circle, fill=black,inner sep=2pt] at (4,3) {};
    \end{tikzpicture}
    \caption{The neighbours of $e$ in $\Gamma_5$ are indicated in black.}\label{fig1}
\end{figure}

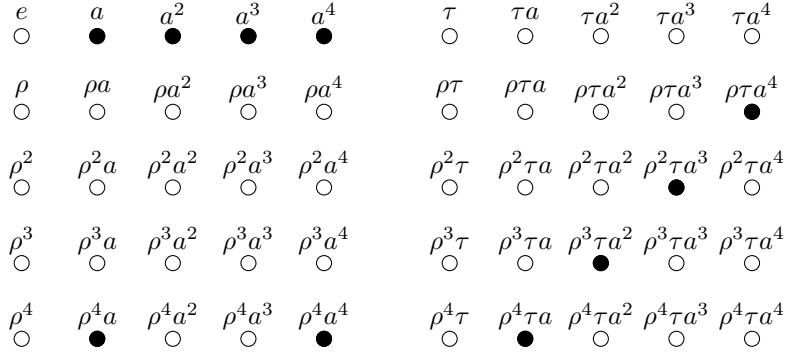
\begin{figure}
    \begin{tikzpicture}
        \foreach \x in {0,...,4}{
        \foreach \y in {0,...,4}{
\node[draw,circle, fill=white,inner sep=2pt] (\x) at (\x,\y) {};
}}
        \foreach \x in {2,...,4}{
        \foreach \y in {2,...,4}{
\node at ($(\x,4-\y)+(0,0.3)$) {$\rho^\y a^\x$};
}}
        \foreach \y in {2,...,4}{
\node at (\y,3.3) {$\rho a^\y$};
\node at (\y,4.3) {$a^\y$};
\node at ($(0,4-\y)+(0,0.3)$) {$\rho^\y$};
\node at ($(1,4-\y)+(0,0.3)$) {$\rho^\y a$};
}
\node at (0,3.3) {$\rho$};
\node at (0,4.3) {$e$};
\node at (1,4.3) {$a$};
\node at (1,3.3) {$\rho a$};
\foreach \x in {1,...,4}
{
\node[draw,circle, fill=black,inner sep=2pt] at (\x,4) {};
}
\node[draw,circle, fill=black,inner sep=2pt] at (1,0) {};
\node[draw,circle, fill=black,inner sep=2pt] at (4,0) {};

    \end{tikzpicture}
    \qquad
      \begin{tikzpicture}
        \foreach \x in {0,...,4}{
        \foreach \y in {0,...,4}{
\node[draw,circle, fill=white,inner sep=2pt] (\x) at (\x,\y) {};
}}
        \foreach \x in {2,...,4}{
        \foreach \y in {2,...,4}{
\node at ($(\x,4-\y)+(0,0.3)$) {$\rho^\y \tau a^\x$};
}}
        \foreach \y in {2,...,4}{
\node at (\y,3.3) {$\rho \tau a^\y$};
\node at (\y,4.3) {$\tau a^\y$};
\node at ($(0,4-\y)+(0,0.3)$) {$\rho^\y \tau$};
\node at ($(1,4-\y)+(0,0.3)$) {$\rho^\y \tau a$};
}
\node at (0,3.3) {$\rho \tau$};
\node at (0,4.3) {$\tau$};
\node at (1,4.3) {$\tau a$};
\node at (1,3.3) {$\rho \tau a$};
\foreach \x in {1,...,4}
{
\node[draw,circle, fill=black,inner sep=2pt] at (\x,\x-1) {};
}
    \end{tikzpicture}
    \caption{The neighbours of $\tau$ in $\Gamma_5$ are indicated in black.}\label{fig2}
\end{figure}

We begin by determining the domination number of these graphs.

\begin{lem}\label{dom-gamma}
    When $k\ge 3$ is odd, $\gamma(\Gamma_k)=k$.
\end{lem}

\begin{proof}
    Each vertex has $2(k-1)+2=2k$ neighbours, so dominates $2k+1$ vertices. Since $$(k-1)(2k+1)=2k^2-k-1<2k^2,$$ any dominating set must contain at least $k$ vertices.

    To show that $k$ vertices suffice, we claim that $D=K'_k$ is a dominating set for $\Gamma_k$. Note that these $k$ vertices all lie in distinct left cosets of $K_k$ (this requires $k$ odd), so  all vertices in each left coset of $K_k$ within $H_k$ are dominated by $D$. That is, all vertices of $H_k$ are dominated by $D$.

    The vertex $\rho^j\tau a^\ell$ is dominated by $\rho^j a^{-j}$ unless $\ell=-j$. If $\ell=-j$ then this vertex is dominated by $\rho^{j-1}a^{1-j}$ using the element $\rho\tau a^{-1}$ of the connection set. Thus every vertex that is not in $H_k$ is also dominated by $D$.

    Thus, $\gamma(\Gamma_k)=k$.
\end{proof}

Next we have a lemma that tells us a lot about the structure of any minimal dominating set.

\begin{lem}\label{cover-diagonals}
   Let $k \ge 3$ be odd, and let $D$ be a dominating set in $\Gamma_k$ of cardinality $k$. Then all elements of $D$ lie in the same coset of $H_k$. Furthermore, one of these $k$ elements lies in each left coset of $K_k$ within that coset of $H_k$.
\end{lem}

\begin{proof}
Suppose $v \in H_k$.
Observe from the connection set that $v$ dominates all of the vertices in $vK_k$, and no other vertices of $H_k$. Similarly, if $v \notin H_k$ then $v$ dominates all of the vertices in $vK_k$, and no other vertices of $vH_k$.

Towards a contradiction, suppose that $D$ were a dominating set in $\Gamma_k$ of cardinality at most $k$, that contains at least one vertex in each of the two cosets of $H_k$.

Since $|D\cap H_k|<k$, there is some left coset of $K_k$ in $H_k$, none of whose vertices is dominated by any vertex of $D \cap H_k$. For concreteness, we may assume this left coset is $\rho^iK_k$. Therefore, all $k$ vertices in $\rho^iK_k$ must be dominated by vertices that do not lie in $H_k$.

Observe from the connection set that for any vertex $v$ of $D$ that lies in $\tau H_k$, the vertices of $H_k$ that it dominates lie in only two of the cosets of $A_k$ (the coset that contains $v\tau$ and the  coset that contains $v\rho\tau$). In particular, since each coset of $A_k$ in $H_k$ intersects each coset of $K_k$ in $H_k$ in a unique point, this implies that any vertex of $D$ that lies in $\tau H_k$ can dominate at most $2$ vertices in $\rho^iK_k$. Thus, we must include at least $\left\lceil\frac{k}{2}\right\rceil$ vertices of $\tau H_k$ in $D$, in order to dominate all of the vertices in $\rho^i K_k$.

An analogous argument shows that since $D$ contains fewer than $k$ vertices that lie in $\tau H_k$, there must be at least $\left\lceil\frac{k}{2}\right\rceil$ vertices of $H_k$ in $D$.

Since $k$ is odd, $\left\lceil\frac{k}{2}\right\rceil=\frac{k+1}{2}$, so this makes a total of at least $k+1$ distinct vertices in $D$, contradicting $|D| \le k$. With this contradiction, we conclude that $D \subset H_k$ or $D \subset \tau H_k$.

If $D \subset vH_k$, there cannot be a left coset of $K_k$ in $vH_k$ that does not contain an element of $D$, or none of its vertices would be dominated by any element of $D$.
\end{proof}

We still need additional information about the structure of minimal dominating sets.

\begin{lem}\label{cover-rows}
    Let $k \ge 5$ be odd, and let $D$ be a dominating set in $\Gamma_k$ of cardinality $k$ with $D \subset H_k$. Then at least one element of $D$ lies in each coset of $A_k$ in $H_k$.
\end{lem}

\begin{proof}
      Using~\Cref{cover-diagonals}, we see that one of the vertices of $D$ lies in each (left) coset of $K_k$ within $H_k$.

    Towards a contradiction, suppose that some coset of $A_k$ in $H_k$ contains no vertex of $D$; say this is the coset $\rho^i A_k$. Consider how the vertices of $\tau H_k$ can be dominated by elements of $D$. Any vertex of $D$ that is in $H_k$ and the coset $\rho^j A_k$ dominates $k-1$ vertices in $\rho^j \tau A_k$, and two vertices in $\rho^{j+1}\tau A_k$. It dominates no other vertices that are not in $H_k$. So in order to dominate all of the vertices in $\rho^i \tau A_k$, we must  have at least $\frac{k+1}{2}$ vertices of $D$ in $\rho^{i-1}A_k$.

    Observe that none of the vertices of $D$ that are in $\rho^{i-1}A_k$ dominate any vertex of $\tau H_k$ that is not in $\rho^{i-1}\tau A_k\cup \rho^i \tau A_k$. So the $k(k-2)$ vertices that are in $\tau H_k -(\rho^{i-1}A_k\cup \rho^i A_k)$ must be dominated by the remaining vertices of $D$. Since we have at least $\frac{k+1}{2}$ vertices of $D$ in $\rho^{i-1}A_k$, there are at most $\frac{k-1}{2}$ of these remaining vertices. Each vertex of $D$ dominates $k+1$ vertices of $\tau H_k$. So these remaining vertices of $D$ can dominate at most $\frac{k-1}{2}(k+1)$ vertices that are in $\tau H_k-(\rho^{i-1}A_k\cup \rho^i A_k)$. 

    Since $D$ is a dominating set, we conclude that $$\frac{k-1}{2}(k+1) \ge k(k-2),$$ which simplifies to $$k^2-4k+1 \le 0.$$ This implies $k(k-4) \le 0;$ since both factors are positive when $k\ge 5$, this inequality is never satisfied, implying that $D$ is not a dominating set. This contradiction completes the proof.
\end{proof}

This information is enough to determine (up to isomorphism) what every dominating set looks like.

\begin{lem}\label{dom-set-types}
    Let $k \ge 5$ be odd, and let $D$ be a dominating set in $\Gamma_k$ of cardinality $k$ that contains $e$ (the identity vertex). Then $D=K_k'$.
\end{lem}

\begin{proof}
    We begin by observing, from~\Cref{cover-diagonals}, that $D \subset H_k$. Also from~\Cref{cover-rows}, there is exactly one vertex of $D$ in each coset of $A_k$ in $H_k$.

    Suppose that the vertex of $D$ in $\rho^i A_k$ is $\rho^i a^j$. In $\rho^{i+1}A_k$, let the vertex of $D$ be $\rho^{i+1} a^\ell$. Then since $\rho^{i+1}\tau a^\ell$ is not dominated by $\rho^{i+1} a^\ell$, it must be dominated by $\rho^i a^j$. Therefore we must have $\ell \in \{j\pm1\}$.

    If $\ell=j+1$, then these two vertices of $D$ lie in the same coset of $K_k$, contradicting~\Cref{cover-diagonals}. So we must have $\ell=j-1$. Applying this inductive argument beginning with the vertex $e$ of $D$ gives $D=K_k',$ as claimed.
\end{proof}

Finally we are ready to consider the eternal domination number for these graphs.

\begin{thm}\label{etdom-bigger}
    Let $k\ge 5$ be odd. Then $\etdom(\Gamma_k) >\gamma(\Gamma_k)$.
\end{thm}

\begin{proof}
    By~\Cref{dom-gamma}, $\gamma(\Gamma_k)=k$, so we need to show that $\etdom(\Gamma_k)>k$.

    We will show that for any dominating set $D$ of $\Gamma_k$ that has cardinality $k$, there exists a vertex $v$ such that guards from $D$ cannot be reconfigured to $D'$ for any dominating set $D'$ that contains $v$.

    Using the vertex-transitivity of $\Gamma_k$, we may assume that $D=K_k'.$ Let $v=\tau$. By~\Cref{dom-set-types} and the vertex-transitivity of $\Gamma_k$, there is a unique dominating set that contains $v$, and it is $\tau K_k'=\{\rho^{-i}\tau a^{-i}: 0 \le i \le k-1\}.$ So we must show that guards cannot reconfigure from $D=K_k'$ to $D'=\tau K_k'$.

    There is a unique vertex of $D$ that dominates $v$: namely, $\rho^{-1}a$. So if there were a reconfiguration, the defender on this vertex must move to $v$. 

    Observe that $D'=\tau K_k'$ also includes the vertex $\rho^{-1}\tau a^{-1},$ and $\rho^{-1}a$ is also the only vertex of $D$ that dominates this vertex. Therefore we cannot reconfigure the defenders from $D$ to $D'$.
\end{proof}

\begin{thm}
    Let $k >1$ be odd. Then $\etdom(\Gamma_k)=\gamma(\Gamma_k)+1$.
\end{thm}

\begin{proof}
    For $k=3$, this was shown in~\cite{BSL2016}; their paper does not include a proof of the claim that $\gamma(\Gamma_3) \le k+1$, but they checked this by computer and it can be verified. For the remainder of the proof, we may assume $k \ge 5$.

    By~\Cref{etdom-bigger}, $\etdom(\Gamma_k)>\gamma(\Gamma_k)=k$, so it suffices to show that there is an eternal dominating configuration whose sets have cardinality $k+1$. We claim that $$\mathcal D=\{g(K_k'\cup \{\tau\}): g \in G_k\}$$ is such a configuration. 

    Using the fact that $\mathcal D$ is invariant under left-multiplication by elements of $G_k$, we may assume without loss of generality that we start with the dominating set $D=K_k'\cup \{\tau\}$. We now consider a variety of possible vertices $v \notin D$. 

    Suppose that $v=\rho^j\tau a^{-j}$ for some $j$. Then moving the defender from $\tau$ to the adjacent vertex $\tau\rho^{-j}a^{-j}=\rho^j\tau a^{-j}=v$ and leaving all other defenders in place reconfigures the defenders to $\rho^ja^{-j}(K_k'\cup \tau) \in \mathcal D$. 

    Next suppose that $v=\rho^j\tau a^\ell$ for some $\ell \neq -j$. In this case, we will reconfigure the defenders to $\tau a^{\ell+j} D \in \mathcal D$. Note this includes $v=\tau a^{\ell+j}\rho^{-j}a^{-j}$.

    For every $i$, we reconfigure the defender from vertex $\rho^{j-i} a^{i-j}$ to the vertex $\rho^{j-i}\tau a^{\ell+i}$. Since $j \neq -\ell$, there is an edge between these vertices. We also move the defender from $\tau$ to $a^{\ell+j}$, which is adjacent since $\ell+j\neq 0$. These are exactly the vertices of $\tau a^{\ell+j}D$.
    
    Finally, suppose that $v=\rho^j a^\ell$ for some $\ell \neq -j$. Take $m$ to be the unique value modulo $k$ such that $2m = -j-\ell$. We will reconfigure the defenders to occupy the vertices of $\rho^{j+m}\tau a^{\ell+m}D \in \mathcal D$.
    
    Since $-(j+m)=\ell+m$, $\rho^{j+m} a^{\ell+m} \in D$ is in the same (left) coset of $K_k$ as $v=\rho^j a^{\ell}$ and therefore dominates it. We reconfigure the defender from $\rho^{j+m} a^{\ell+m}$ to $\rho^j a^\ell$. Also, for each $1 \le i \le k-1$, we reconfigure the defender from $\rho^{j+m+i}a^{-j-m-i}$ to $\rho^{j+m+i}\tau a^{\ell+m+i}$. We can do this as long as $-j-m-i \neq \ell+m+i$; that is, $2i+2m \neq -j-\ell$. Since $2m=-j-\ell$, this is equivalent to $2i\neq 0$; since $k$ is odd, this is equivalent to $i \neq 0$, and therefore does always work. Finally, we reconfigure the defender from $\tau$ to $\rho^{j+m}\tau a^{\ell+m}$. This edge exists because these vertices are in the same left coset of $K_k$ (as noted above, $-(j+m)=\ell+m$). These are the vertices of $\rho^{j+m}\tau a^{\ell+m}D$.  
\end{proof}

\section{Larger gaps}

For this section, we let $b$ be an element of order $2$ that centralises $G_k$, so $\langle G_k,b \rangle \cong D_{2k} \times C_{2k}$ has order $4k^2$. Our Cayley graphs will be defined on this group.

\begin{defn}
    For any odd $k\ge 3$, we define the family of Cayley graphs $\Delta_k$ as $\Delta_k=\Cay{\langle G_k,b \rangle}{T_k}$, where $T_k=(S_k\setminus\{\rho a,\rho^{-1}a^{-1}\}) \cup \{\rho ab,\rho^{-1}a^{-1}b\}$.
\end{defn}

In~\Cref{fig3} we provide an image showing the neighbours of the identity vertex $e$ in $\Delta_5$, similar to the image for $\Gamma_5$ that was provided earlier.

\begin{figure}
    \begin{tikzpicture}
        \foreach \x in {0,...,4}{
        \foreach \y in {6,...,10}{
\node[draw,circle, fill=white,inner sep=2pt] (\x) at (1.2*\x,\y) {};
}}
        \foreach \x in {2,...,4}{
        \foreach \y in {2,...,4}{
\node at ($(1.2*\x,10-\y)+(0,0.3)$) {$\rho^\y a^\x$};
}}
        \foreach \y in {2,...,4}{
\node at (1.2*\y,9.3) {$\rho a^\y$};
\node at (1.2*\y,10.3) {$a^\y$};
\node at ($(0,10-\y)+(0,0.3)$) {$\rho^\y$};
\node at ($(1.2,10-\y)+(0,0.3)$) {$\rho^\y a$};
}
\node at (0,9.3) {$\rho$};
\node at (0,10.3) {$e$};
\node at (1.2,10.3) {$a$};
\node at (1.2,9.3) {$\rho a$};
\foreach \x in {2,3}
{
\node[draw,circle, fill=black,inner sep=2pt] at (1.2*\x,10-\x) {};
}
        \foreach \x in {0,...,4}{
        \foreach \y in {0,...,4}{
\node[draw,circle, fill=white,inner sep=2pt] (\x) at (1.2*\x,\y) {};
}}
        \foreach \x in {2,...,4}{
        \foreach \y in {2,...,4}{
\node at ($(1.2*\x,4-\y)+(0,0.3)$) {$b\rho^\y a^\x$};
}}
        \foreach \y in {2,...,4}{
\node at (1.2*\y,3.3) {$b\rho a^\y$};
\node at (1.2*\y,4.3) {$ba^\y$};
\node at ($(0,4-\y)+(0,0.3)$) {$b\rho^\y$};
\node at ($(1.2,4-\y)+(0,0.3)$) {$b\rho^\y a$};
}
\node at (0,3.3) {$b\rho$};
\node at (0,4.3) {$b$};
\node at (1.2,4.3) {$ba$};
\node at (1.2,3.3) {$b\rho a$};
\node[draw,circle, fill=black,inner sep=2pt] at (1.2,3) {};
\node[draw,circle, fill=black,inner sep=2pt] at (1.2*4,0) {};

    \end{tikzpicture}
    \qquad
      \begin{tikzpicture}
        \foreach \x in {0,...,4}{
        \foreach \y in {0,...,4}{
\node[draw,circle, fill=white,inner sep=2pt] (\x) at (1.2*\x,6+\y) {};
}}
        \foreach \x in {2,...,4}{
        \foreach \y in {2,...,4}{
\node at ($(1.2*\x,10-\y)+(0,0.3)$) {$\rho^\y \tau a^\x$};
}}
        \foreach \y in {2,...,4}{
\node at (1.2*\y,9.3) {$\rho \tau a^\y$};
\node at (1.2*\y,10.3) {$\tau a^\y$};
\node at ($(0,10-\y)+(0,0.3)$) {$\rho^\y \tau$};
\node at ($(1.2,10-\y)+(0,0.3)$) {$\rho^\y \tau a$};
}
\node at (0,9.3) {$\rho \tau$};
\node at (0,10.3) {$\tau$};
\node at (1.2,10.3) {$\tau a$};
\node at (1.2,9.3) {$\rho \tau a$};
\foreach \x in {1,...,4}
{
\node[draw,circle, fill=black,inner sep=2pt] at (1.2\x,10) {};
}
\node[draw,circle, fill=black,inner sep=2pt] at (1.2,9) {};
\node[draw,circle, fill=black,inner sep=2pt] at (1.2*4,9) {};
        \foreach \x in {0,...,4}{
        \foreach \y in {0,...,4}{
\node[draw,circle, fill=white,inner sep=2pt] (\x) at (1.2*\x,\y) {};
}}
        \foreach \x in {2,4}{
        \foreach \y in {2,...,4}{
\node at ($(1.2*\x,4-\y)+(0,0.3)$) {$b\rho^\y \tau a^\x$};
}}
        \foreach \y in {2,...,4}{
\node at ($(1.2*3,4-\y)+(0,0.3)$) {$b\rho^\y \tau a^3$};
\node at (1.2*\y,3.3) {$b\rho \tau a^\y$};
\node at (1.2*\y,4.3) {$b\tau a^\y$};
\node at ($(0,4-\y)+(0,0.3)$) {$b\rho^\y \tau$};
\node at ($(1.2,4-\y)+(0,0.3)$) {$b\rho^\y \tau a$};
}
\node at (0,3.3) {$b\rho \tau$};
\node at (0,4.3) {$b\tau$};
\node at (1.2,4.3) {$b\tau a$};
\node at (1.2,3.3) {$b\rho \tau a$};

    \end{tikzpicture}
    \caption{The neighbours of $e$ in $\Delta_5$ are indicated in black.}\label{fig3}
\end{figure}
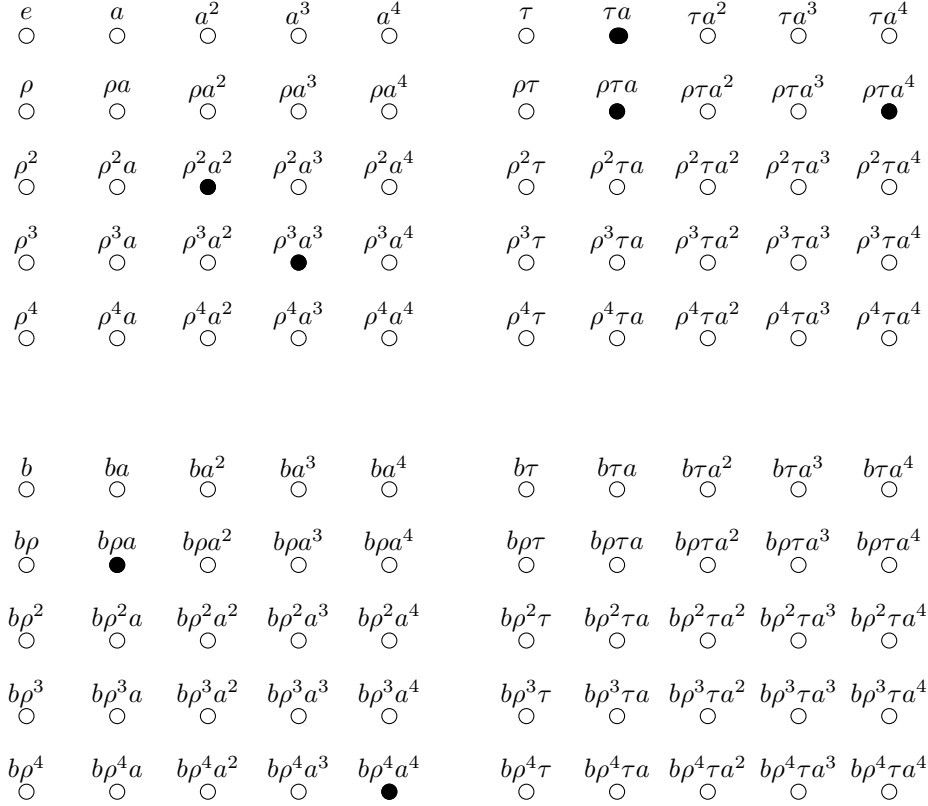

As before, cosets of interest will include cosets of $A_k$ and left cosets of $K_k$; now we also have two cosets of $G_k$, and two cosets of $\langle H_k, b\rangle$. Finally, we will be interested in left cosets of $$B_k=\langle A_k, \tau \rangle.$$

\begin{lem}\label{dom-delta}
    When $k \ge 3$ is odd, $\gamma(\Delta_k)=2k$.
\end{lem}

\begin{proof}
    Observe that $\langle K_k',b\rangle$ is a dominating set for $\Delta_k$ and has $2k$ vertices. (In $\Gamma_k$, $K_k'$ was a dominating set for all elements of $G_k$; the vertex $v$ no longer dominates $v\rho a$ or $v\rho^{-1}a^{-1}$ but these are now dominated by $vb$. A similar argument holds in the other coset of $G_k$.)

    Since $|T_k|=2k$, each vertex dominates $2k+1$ vertices. If $D$ is a set with $2k-1$ vertices, then the number of vertices that $D$ can dominate is therefore at most $(2k-1)(2k+1)=4k^2-1$ vertices, which is fewer than the total number of vertices. Therefore a dominating set must have cardinality at least $2k$.
\end{proof}

We now work to understand the dominating sets for $\Delta_k$ that contain $2k$ or $2k+1$ vertices. As a first step, we determine how they interact with the two cosets of $G_k$; that is, the top and bottom halves of our diagram.

\begin{lem}\label{balanced-copies}
    Suppose $k >5$ is odd. A dominating set of cardinality at most $2k+1$ for $\Delta_k$ must contain at least $k$ vertices in each coset of $G_k$.
\end{lem}

\begin{proof}
    Let $D$ be a dominating set of cardinality at most $2k+1$ for $\Delta_k$. Note that by~\Cref{dom-delta}, the cardinality of $D$ must be either $2k$ or $2k+1$.
    
    Towards a contradiction, suppose that $D$ contains $\ell \le k-1$ vertices in one of the cosets of $G_k$, without loss of generality $G_k$ itself. Then in $G_k$, each of these $\ell$ vertices dominates at most $2k-1$ vertices, while the vertices in $bG_k$ (of which there are at most $2k+1-\ell$) each dominates just two vertices in $G_k$. So the total number of vertices in $G_k$ that are dominated is at most $\ell(2k-1)+4k+2-2\ell=(2k-3)\ell +4k+2$. Since $\ell \le k-1$, this is at most $$(2k-3)(k-1)+4k+2=2k^2-k+5<2k^2$$ since $k >5$. This is the desired contradiction.
\end{proof}

For our next results we define the concept of overlap. 

\begin{defn}
    Let $X$ be a set of vertices of a graph $\Gamma$, and let $\Gamma[v]$ denote the closed neighbourhood of a vertex $v$ of $\Gamma$, and $\Gamma[X]=\bigcup_{v \in X}\Gamma[v]$. The \emph{overlap} of $X$ is $$\ov(X)=\left(\sum_{v \in X}|\Gamma[v]|\right)-|\Gamma[X]|.$$

    Analogously, we can define the overlap within a subset of the vertices of $\Gamma$: if $Y$ is a subset of the vertices of $\Gamma$ then $$\ov_Y(X)=\left(\sum_{v \in X}|\Gamma[v] \cap Y|\right)-|\Gamma[X]\cap Y|.$$
\end{defn}

Intuitively, we can think of the overlap of a dominating set of vertices as measuring the flexibility we may have in adjusting that dominating set: it tells us something about how many vertices are dominated by more than one vertex of the dominating set. An efficient dominating set has an overlap of $0$. Overlap together with inclusion-exclusion arguments limit the number of vertices that can be dominated by a set of vertices.

\begin{lem}\label{overlap-Gk}
    Suppose $k \ge 3$ is odd, and $D$ is a dominating set for $\Delta_k$ that has $k$ vertices in $b^iG_k$, and either $k$ or $k+1$ vertices in $b^{1-i}G_k$, for some $i \in \{0,1\}$. Then $\ov_{b^iG_k}(D) \le k+2$ and $\ov_{b^{1-i}G_k}(D) \le 3k-1$.
\end{lem}

\begin{proof}
    Taking $Y=b^iG_k$ and $v \in D \cap b^iG_k$, we have $|\Delta_k[v] \cap Y|=2k-1$, while for $v \in D \cap b^{1-i}G_k$ we have $|\Delta_k[v]\cap Y|=2$. So $$\sum_{v \in D}|\Delta_k[v]\cap Y| \le k(2k-1)+(k+1)2=2k^2+k+2,$$ while since $D$ is a dominating set, $|\Delta_k[D] \cap Y| = |Y|=2k^2$. Thus, $\ov_Y(D) \le k+2$. 

     Similarly, taking $Y=b^{1-i}G_k$ and $v \in D \cap b^{1-i}G_k$, we have $|\Delta_k[v] \cap Y|=2k-1$, while for $v \in D \cap b^{i}G_k$ we have $|\Delta_k[v]\cap Y|=2$. So $$\sum_{v \in D}|\Delta_k[v]\cap Y| \le (k+1)(2k-1)+(k)2=2k^2+3k-1,$$ while since $D$ is a dominating set, $|\Delta_k[D] \cap Y| = |Y|=2k^2$. Thus, $\ov_Y(D) \le 3k-1$.
  \end{proof}

  The next lemma seems a bit technical but is intuitively fairly straightforward: if we increase the number of vertices in our set, and/or the set of neighbours we are examining for overlaps, this can't decrease the total overlap we find.

\begin{lem}\label{overlapping-subset}
    If $\Gamma$ is a graph, $D' \subseteq D$ and $Y' \subseteq Y$ are subsets of vertices, and $D$ is a dominating set, then $\ov_{Y'}(D') \le \ov_Y(D)$.
\end{lem}

\begin{proof}
    Observe that $\sum_{v \in D-D'} |\Gamma[v]\cap Y|$ must count every vertex of $Y$ that is dominated by something in $D-D'$. Also, $\sum_{v \in D'}|\Gamma[v]\cap(Y-Y')|$ counts every vertex of $Y-Y'$ that is dominated by something in $D'$. Since $D$ is a dominating set, every vertex of $Y-Y'$ is counted in one of these sums. Finally, $\Gamma[D']\cap Y'$ includes every vertex of $Y'$ that is dominated by something in $D'$, so again since $D$ is a dominating set, every vertex of $Y'$ is included either in the first sum or this set.

    We conclude that $$\left(\sum_{v \in D-D'} |\Gamma[v]\cap Y|\right)+\left(\sum_{v \in D'}|\Gamma[v]\cap(Y-Y')|\right)+|\Gamma[D']\cap Y'| \ge |Y|.$$
    Notice that $$\left(\sum_{v \in D-D'} |\Gamma[v]\cap Y|\right)+\left(\sum_{v \in D'}|\Gamma[v]\cap(Y-Y')|\right)=\left(\sum_{v \in D}|\Gamma[v]\cap Y|\right)-\left(\sum_{v \in D'} |\Gamma[v]\cap Y'|\right),$$ so substituting this into our original inequality gives
    $$\left(\sum_{v \in D}|\Gamma[v]\cap Y|\right)-\left(\sum_{v \in D'} |\Gamma[v]\cap Y'|\right)+|\Gamma[D']\cap Y'| \ge |Y|.$$
    Rearranging gives $$\left(\sum_{v \in D}|\Gamma[v]\cap Y|\right)-|Y|\ge \left(\sum_{v \in D'} |\Gamma[v]\cap Y'|\right)-|\Gamma[D']\cap Y'|;$$ since $D$ is a dominating set we have $\Gamma[D]\cap Y=Y$, so this is equivalent to $\ov_Y(D)\ge \ov_{Y'}(D')$.
\end{proof}

The bounds we present in our next lemma are not optimal, but in this context it seems more useful to have a clear proof of a weaker bound, than a finicky proof of a tighter bound. We will only actually use this lemma when $i$ and $j$ differ by exactly $2$ modulo $k$, but since the same proof holds as long as $i$ and $j$ are distinct we provide the more general statement.

\begin{lem}\label{rows-small}
    Suppose $k \ge 3$ is odd, and $D$ is a dominating set for $\Delta_k$ that has $k$ vertices of $G_k$ and either $k$ or $k+1$ vertices of $bG_k$.

    Then for any $0 \le i<j \le k-1$ we have $$|D\cap (\rho^iB_k\cup \rho^jB_k)| \le 5,$$ and
    $$|D \cap b(\rho^iB_k\cup \rho^jB_k)| \le 6.$$
\end{lem}

\begin{proof}
    By~\Cref{overlap-Gk}, $\ov_{G_k}(D) \le k+2$. 

     Suppose that $|D\cap (\rho^iB_k\cup \rho^jB_k)|\ge 6.$  Note that for every $\ell$, any vertex in $\rho^\ell B_k$ dominates $k$ vertices in $\rho^\ell B_k$. Since there are $2k$ vertices in $\rho^\ell B_k$, this implies that $$\ov_{\rho^iB_k\cup \rho^jB_k}(D\cap (\rho^iB_k\cup \rho^jB_k)) \ge 6k-4k=2k.$$
    Since $k\ge 3$, we have $2k>k+2\ge\ov_{G_k}(D)$. This is a contradiction using~\Cref{overlapping-subset}, as $D\cap (\rho^iB_k\cup \rho^jB_k) \subset D$ and $\rho^iB_k\cup \rho^jB_k \subset G_k$. So there can be at most $5$ vertices in $D\cap (\rho^iB_k\cup \rho^jB_k)$.
    
    
    Similarly, if $D$ has $k+1$ vertices in $bG_k$, then by~\Cref{overlap-Gk}, $\ov_{bG_k}(D) \le 3k-1$.
    
      Suppose that $|D\cap b(\rho^iB_k\cup \rho^jB_k)|\ge 7.$ Note that for every $\ell$, any vertex in $b\rho^\ell B_k$ dominates $k$ vertices in $b\rho^\ell B_k$. Since there are $2k$ vertices in $b\rho^\ell B_k$, this implies that $$\ov_{b(\rho^iB_k\cup \rho^jB_k)}(D\cap b(\rho^iB_k\cup \rho^jB_k)) \ge 7k-4k=3k.$$
    Since $3k>3k-1$, this is a contradiction using~\Cref{overlapping-subset}, as $D \cap b(\rho^iB_k\cup \rho^jB_k) \subset D$ and $b(\rho^iB_k\cup \rho^jB_k) \subset bG_k$. So there can be at most $6$ vertices in $D\cap b(\rho^iB_k\cup \rho^jB_k)$.
\end{proof}

We use these bounds to further understand the structure of dominating sets in $\Delta_k$ that are either minimal, or within one of being minimal. 

As a first step, we consider how these dominating sets interact with left cosets of $B_k$; that is, the rows in \Cref{fig3}. Although the following result may be true for some smaller values of $k$, our goal is simply to find an infinite family so we have not attempted to verify what happens when $3 \le k < 17$.

\begin{lem}\label{rows-covered}
    Suppose $k \ge 17$ is odd. A dominating set of cardinality at most $2k+1$ for $\Delta_k$ must contain at least one vertex in each left coset of $B_k$.
\end{lem}

\begin{proof}
    Let $D$ be a dominating set of cardinality at most $2k+1$. By~\Cref{balanced-copies}, $D$ must have $k$ vertices in one coset of $G_k$ and either $k$ or $k+1$ vertices in the other coset. Without loss of generality, we may assume that $G_k$ contains $k$ vertices of $D$.

    Suppose that there is some $i$ such that in $G_k$, $\rho^iB_k$ has no vertices of $D$. Each vertex of $D \cap G_k$ dominates one vertex of $\rho^iB_k$, unless it lies in either $\rho^{i-1}A_k$ or $\rho^{i+1}\tau A_k$, in which case it dominates $3$ vertices of $\rho^iB_k$. By~\Cref{rows-small}, there can be at most $5$ vertices in $D\cap G_k$ that dominate $3$ vertices of $\rho^i A_k$. This implies that in total, the vertices of $D\cap G_k$ dominate at most $5(3)+(k-5)1=k+10$ vertices of $\rho^iB_k$.
    
    Since there are a total of $2k$ vertices in $\rho^iB_k$, the remaining vertices (of which there are at least $k-10$) must be dominated by vertices in $D \cap bG_k$. Observe that a vertex of $D\cap bG_k$ dominates a vertex in $\rho^iB_k$ only if it lies in either $b\rho^{i-1}B_k$, or $b\rho^{i+1}B_k$, and in either case it dominates exactly one vertex of $\rho^i A_k$. Thus there must be at least $k-10$ vertices of $D$ in $b(\rho^{i-1}B_k\cup \rho^{i+1}B_k)$. However,~\Cref{rows-small} tells us that there are at most $6$ vertices in this set, so $k-10\le 6$, contradicting $k \ge 17$.

    Since the same argument can also be applied to $bG_k$ if it has $k$ vertices of $D$, we conclude that in any coset of $G_k$ that has $k$ vertices, there is a vertex in each coset of $A_k$ within that coset of $G_k$ (and therefore exactly one such vertex).

    We must now consider the possibility that $D$ has $2k+1$ vertices. Without loss of generality using~\Cref{balanced-copies} and symmetry, we may assume that $|G_k \cap D|=k$ and $|bG_k\cap D|=k+1$. Suppose that $D\cap b\rho^iB_k=\emptyset$. 
    Again observe that the only vertices in $D$ that can dominate more than one vertex of $b\rho^iB_k$ are the vertices of $D$ that lie in $b\rho^{i-1}B_k\cup b\rho^{i+1}B_k$. By~\Cref{rows-small}, there are at most $6$ of these. Furthermore, each of these dominates exactly $3$ vertices of $b\rho^iB_k$. So the number of vertices of $b\rho^iB_k$ that are dominated by vertices of $D \cap bG_k$ is at most $6(3)+(k+1-6)=k+13$. 

 We conclude that at least $k-13 \ge 4$ vertices of $D\cap G_k$ must dominate vertices in $b\rho^iB_k$. But this implies that each of these vertices lies in $\rho^{i-1}B_k\cup \rho^{i+1}B_k$. We have concluded above that there is exactly one vertex of $D$ in each set $\rho^jB_k$, which means there are exactly two vertices of $D$ in $\rho^{i-1}B_k\cup \rho^{i+1}B_k$, and each dominates exactly one vertex of $b\rho^iB_k$, so they do not dominate at least $4$ such vertices. This is the contradiction that completes the proof.
\end{proof}

We still need a greater understanding of what these fairly small dominating sets look like, in order to determine whether or not they can allow eternal domination of the graph. We next consider how they interact with the cosets of $\langle H_k, b \rangle$.
These are the left and right halves of \Cref{fig3}.

\begin{lem}\label{same-coset}
    Suppose $k \ge 17$ is odd. A dominating set $D$ of cardinality at most $2k+1$ for $\Delta_k$ must contain at least $2k$ vertices in the same coset of $\langle H_k,b\rangle$. Furthermore, if there is a vertex of $D$ in the other coset of $\langle H_k,b\rangle$, then it is in a coset of $G_k$ that contains $k+1$ vertices of $D$.
\end{lem}

\begin{proof}
    By~\Cref{balanced-copies}, there is a coset of $G_k$ that has $k$ elements of $D$. Without loss of generality, we may assume this is $G_k$ itself. Towards a contradiction, suppose that at least one of these $k$ vertices lies in each of the two cosets of $H_k$ in $G_k$. In particular, this means that some coset of $K_k$ that lies in $H_k$ contains no vertex of $D$ (and the same is true for some coset of $K_k$ that lies in $\tau H_k$). Say this coset is $\rho^iK_k$.

    Any vertex of $D$ that lies in $bG_k$ can dominate at most $2$ vertices of $\rho^iK_k$; furthermore, all vertices from $bG_k$ that dominate vertices in $\rho^iK_k$ must themselves lie in $\rho^ibK_k$. By~\Cref{overlap-Gk}, $\ov_{bG_k}(D) \le 3k-1$. 
    
    We now make an argument similar to those in the proof of~\Cref{overlap-Gk}, with respect to cosets of $K_k$. We plan to show that any coset of $K_k$ contains at most $4$ vertices of $D$. Consider the coset $vK_k$. By~\Cref{overlap-Gk}, $\ov_{vG_k}(D) \le 3k-1$. Suppose $|D\cap vK_k| \ge 5$. Since every vertex in $vK_k$ dominates $k-2$ vertices in $vK_k$, this implies that $$\ov_{vK_k}(D \cap vK_k) \ge 5(k-2)-k=4k-10.$$ Since $k>9$ we have $4k-10>3k-1$, a contradiction using~\Cref{overlapping-subset}. This establishes that each coset of $K_k$ contains at most $4$ vertices of $D$, as claimed. In fact, putting slightly more thought into this analysis leads to the conclusion that if there are $4$ vertices of $D$ in one coset of $K_k$ in $bG_k$, then there is at most one vertex of $D$ in any other coset of $K_k$ in $bG_k$, and more generally the total number of ``extra" vertices (beyond one) of $D$ in all cosets of $K_k$ within $bG_k$ is at most $3$.

We conclude that vertices of $D \cap bG_k$ dominate at most $8$ vertices of $\rho^iK_k$. So the remaining at least $k-8$ vertices of $\rho^iK_k$ must be dominated by vertices of $D \cap G_k$. 

Since $D \cap \rho^iK_k=\emptyset$, no vertex from $D \cap H_k$ dominates any vertex of $\rho^iK_k$. Any vertex in $\tau H_k$ dominates at most $2$ vertices of $\rho^iK_k$. So there must be at least $(k-8)/2$ vertices in $D \cap \tau H_k$.
 Since $k \ge 14$ we have $(k-8)/2\ge 3$. So there are at least $3$ vertices in $D\cap \tau H_k$, which  means that there are at most $k-3$ vertices in $D\cap H_k$, since there are $k$ vertices in $D\cap G_k$. In particular, this means that there are actually at least $3$ cosets of $K_k$ in $H_k$ that have no vertices of $D$. We repeat the same argument with respect to these other cosets that have no vertices of $D$. If there were only $(k-8)/2$ vertices in $D\cap \tau H_k$, then each dominate $2$ vertices in at most 2 of the empty cosets of $K_k$, so together with vertices from $D \cap bG_k$ could account for dominating all vertices of $\rho^iK_k$, and $k-8+2=k-6$ of the vertices from a second coset that has no vertices of $D$. In the third coset that has no vertices of $D$, the vertices of $D\cap bG_k$ could dominate $2$ and the $(k-8)/2$ vertices in $D\cap \tau H_k$ each dominates only one, leaving only $(k-4)/2$ vertices dominated, so at least $(k+4)/2$ that need to be dominated. This means that we need another at least $(k+4)/4$ vertices in $D\cap \tau H_k$ to complete the domination of this coset of $K_k$. We conclude that $D$ has at least $(k-8)/2+(k+4)/4=(3k-12)/4$ vertices in $D \cap \tau H_k$. 

    We can make exactly the same argument with respect to empty cosets of $K_k$ in  $\tau H_k$, so there must also be at least $(3k-12)/4$ vertices in $D$ that are in $H_k$. Since $G_k$ contains $k$ vertices of $D$, we conclude $(3k-12)/2 \le k$, implying $k \le 12$, a contradiction.

    This contradiction implies that all $k$ vertices in $G_k$ that lie in $D$, must lie in a single coset of $H_k$; without loss of generality and for clarity, suppose that they lie in $H_k$. Observe that these vertices dominate exactly $k(k-2)=k^2-2k$ of the vertices in $H_k$, so the other $2k$ vertices of $H_k$ must be dominated by other vertices of $D$, all of which lie in $bG_k$. The only vertices of $D$ in  $bG_k$ that dominate any vertices in $H_k$ are vertices in $bH_k$, and each of these dominates exactly $2$ vertices in $H_k$, so there must be at least $k$ vertices in $bG_k$ that lie in $b H_k$. This makes a total of $2k$ vertices in $\langle H_k, b\rangle$, as claimed, including $k$ from each coset of $G_k$.
\end{proof}

We still need more information, particularly about how the elements of the dominating set lie within the coset of $\langle H_k,b\rangle$ that contains almost all of them.

\begin{lem}\label{delta-cover-diags}
    Suppose $k \ge 17$ is odd and $D$ is a dominating set for $\Delta_k$ with at most $2k+1$ vertices. In the coset of $\langle H_k, b\rangle$ that contains at least $2k$ vertices of $D$, there is at least one vertex of $D$ in each left coset of $K_k$.
\end{lem}

\begin{proof}
    Consider first a coset of $G_k$ that has at most $k$ vertices of $D$; without loss of generality suppose this is $G_k$. The vertices of $D\cap G_k$ all lie in one coset of $H_k$; without loss of generality, we may assume this is $H_k$. The vertices of $D$ in $bG_k$ dominate at most $2(k+1)=2k+2$ vertices from $H_k$. If two of the vertices in $D \cap H_k$ lie in the same left coset of $K_k$, then they may dominate all $k$ vertices in this left coset but no other vertices in $H_k$, while any other of the $k-2$ vertices of $D\cap H_k$ dominates at most $k-2$ vertices of $H_k$, all of which lie in its own left coset of $K_k$. Thus, the total number of vertices in $H_k$ dominated by the vertices of $D\cap H_k$ would be at most $(k-2)(k-2)+k=k^2-3k+4$. Adding $2k+2$ to this gives a maximum of $k^2-k+6<k^2$ vertices of $H_k$ dominated. So in order to dominate all vertices of $H_k$, the vertices of $D$ that lie in $H_k$ must all be in distinct left cosets of $K_k$.

    Now in each left coset of $K_k$ in $H_k$, say $\rho^iK_k$, the element of $D$ in $\rho^iK_k$ dominates $k-2$ of the vertices of $\rho^iK_k$, so there remain exactly $2$ vertices that need to be dominated by vertices of $D$ that lie in $bG_k$. In order to dominate the other two vertices in $\rho^iK_k$, the dominating vertex must lie in $\rho^ibK_k$. Thus for each $i$, there must be an element of $D$ in $\rho^ibK_k$ also.
\end{proof}

We are finally ready to show that there is very little flexibility in the structure of a dominating set for $\Delta_k$ that has at most $2k+1$ vertices.

\begin{prop}\label{dom-set-structure}
    Suppose $k \ge 17$ is odd and $D$ is a dominating set for $\Delta_k$ with at most $2k+1$ vertices that includes exactly $k$ vertices of $G_k$, one of which is $e$. Then $D$ contains $\langle K_k',b\rangle$.
\end{prop}

\begin{proof}
By~\Cref{same-coset}, since $G_k$ has only $k$ vertices of $D$, all of these must lie in the same coset of $\langle H_k, b\rangle$. Since $e \in D\cap G_k$ and $e \in H_k$, this means that all $k$ of these vertices of $D$ lie in $\langle H_k,b\rangle \cap G_k=H_k$. By~\Cref{same-coset}, there are another $k$ or $k+1$ vertices in $bH_k$, and if there are $k$ then there may be one more in $b\tau H_k$.

We begin by showing that whenever two vertices in consecutive cosets of $A_k$ (in $H_k$) do not lie in the same coset of $K_k'$, this forces $D$ to have a vertex in $b\tau H_k$. Accordingly, suppose that the vertex of $D$ in $\rho^iA_k$ (this exists by~\Cref{rows-covered} and the fact that there all $k$ vertices of $G_k \cap D$ lie in $H_k$) is in a different coset of $H_k$ than the vertex of $D$ in $\rho^{i+1}A_k$: say these vertices are $\rho^i a^j$ and $\rho^{i+1}a^{j-\ell}$, so by assumption, $\ell \neq 1$. Consider what vertex of $D$ can dominate $\rho^{i+1}\tau a^{j-\ell}$. This is not dominated by $\rho^{i+1}a^{j-\ell}$ and this is the only vertex of $D$ in $\rho^{i+1}A_k$; furthermore, there are no vertices of $D$ in $\tau H_k$, and no vertex of $bH_k$ is adjacent to any vertex of $\tau H_k$. Therefore, the only options for dominating this vertex are $\rho^i a^{j-\ell+1}$, $\rho^i a^{j-\ell-1}$, $\rho^{i+2}\tau a^{j-\ell-1} b$, or $\rho^i \tau a^{j-\ell+1}b$. Since $\ell\neq 1$ and $\rho^i a^j$ is the only vertex in $D \cap \rho^iA_k$, the first of these is not possible. By~\Cref{delta-cover-diags}, since there are only $k$ vertices of $D$ in $H_k$, there must be exactly one in each coset of $K_k$; since, $\rho^i a^{j-\ell-1}$ is in the same coset of $K_k$ as $\rho^{i+1}a^{j-\ell}$, it cannot also lie in $D$. Thus, either $\rho^{i+2}\tau a^{j-\ell-1} b$, or $\rho^i \tau a^{j-\ell+1}b$ lies in $D$, as claimed.

It is clear from the above argument that if $K_k'$ does not lie within $D$, then the vertices of $D\cap H_k$ must lie within at least two distinct cosets of $K_k'$. Following the implications of that argument, a single vertex in $b\tau H_k$ cannot possibly dominate all of the necessary vertices if there are more than two distinct cosets, or if the vertices in each coset aren't all consecutive. This leaves only the possibility that the vertices in some set of consecutive rows are all shifted into a different coset of $K_k'$. But it is not possible to perform such a shift and still ensure that every coset of $K_k$ in $H_k$ meets $D$. The only remaining possibility is that $D \cap H_k=K_k'$.

Now $D$ has $k$ or $k+1$ additional vertices, at least $k$ of which lie in $bH_k$, with one possible additional vertex in either $bH_k$ or $b\tau H_k$. These vertices must dominate all vertices of $H_k$ that are not dominated by any vertex of $K_k'$ (specifically, the vertices $v\rho a$ and $v\rho^{-1}a^{-1}$ for each $v \in K_k'$, i.e. $\rho^{-i+1}a^{i+1}$ and $\rho^{-i-1}a^{i-1}$ for each $i$). Since any possible vertex in $b \tau H_k$ does not help with this, we see that either every vertex of $bK_k'$ is in $D$ and we are done, or all but one of the vertices of $bK_k'$ are in $D$, and the two vertices of $H_k$ that would be dominated by the missing vertex of $bK_k'$ are instead dominated by two distinct vertices of $bH_k\cap D$. Without loss of generality, suppose $b \notin D$, so we must use $\rho^2a^2b$ and $\rho^{-2}a^{-2}b$ to dominate $\rho a$ and $\rho^{-1}a^{-1}$. But then there is no vertex of $D$ in $bB_k$, contradicting~\Cref{rows-covered}. This completes the proof.
\end{proof}

We restate essentially this result in a more useful form.

\begin{cor}\label{K'b-cosets}
Suppose $k \ge 17$ is odd and $D$ is a dominating set for $\Delta_k$ with at most $2k+1$ vertices. Then $D$ contains a left coset of $\langle K_k',b\rangle$.
\end{cor}

\begin{proof}
    By~\Cref{balanced-copies}, $D$ must contain at least $k$ vertices in each coset of $G_k$, so since it has a total of either $2k$ or $2k+1$ vertices, it must contain exactly $k$ vertices in at least one of the cosets of $G_k$. If $g\in D$ is a vertex in such a coset, then $g^{-1}D$ is a dominating set for $\Delta_k$ with at most $2k+1$ vertices that includes exactly $k$ vertices of $G_k$, one of which is $e$, so by~\Cref{dom-set-structure}, $g^{-1}D$ contains $\langle K_k',b\rangle$. Thus $D$ contains $g\langle K_k',b\rangle$.
\end{proof}

We now examine possibilities for moving the defenders from one coset of $\langle H_k,b\rangle$ to the other.

\begin{lem}\label{moving}
Suppose $k \ge 3$ is odd. In $\Delta_k$, it is not possible to reconfigure $k$ defenders from the vertices of a left coset of $K_k'$ in one coset of $\langle H_k,b\rangle$, to the vertices of a left coset of $K_k'$ in the other coset of $\langle H_k,b\rangle$.
\end{lem}

\begin{proof}
Without loss of generality, using the vertex-transitivity of $\Delta_k$, assume that the defenders begin on the vertices of $K_k'$. Observe that the neighbours of these vertices in $bG_k$ all lie in $\langle H_k,b\rangle$, so the vertices the defenders are trying to move to must lie in $G_k$. Say the vertices they are trying to move to are the vertices of $\rho^i\tau K_k'=\{\rho^{i+j}\tau a^{j}: 0 \le j \le k-1\}$.

Observe that since $k$ is odd, there is a unique value of $j$ such that $\rho^{i+j}a^j \in K_k'$, namely the value for which $i+j=-j$; that is, $2j=-i$. For this value of $j$, the vertex $\rho^{i+j}\tau a^j$ has only one neighbour in $K_k'$, and that neighbour is $\rho^{i+j-1}a^{-i-j+1}$. So a defender must move from $\rho^{i+j-1}a^{-i-j+1}$ to $\rho^{i+j}\tau a^j$. 

Since both the initial and final configuration of defenders has one defender in each coset of $A_k$ in $G_k$, and by considering the edges available, it is straightforward to conclude that for every $m$, the defender from $\rho^ma^{-m}$ must move to $\rho^{m+1}\tau a^{m+1-i}$. However, an edge between these vertices exists only if $-m+1=m+1-i$ or $-m-1=m+1-i$; that is, $2m=i$ or $2m=i-2$. Each of these possibilities has a unique solution for $m$ since $k$ is odd. Therefore only $2$ of the defenders can move to the appropriate vertices. This contradiction completes the proof.
\end{proof}

The preceding result is very powerful in terms of understanding eternal domination of these graphs.

\begin{cor}\label{cant-move}
Suppose $k \ge 17$ is odd and $D$ is a dominating set for $\Delta_k$ with at most $2k+1$ vertices that contains a left coset of $\langle K_k',b\rangle$ in $\tau^i\langle H_k,b\rangle$ for some $i \in \{0,1\}$. Then it is not possible to reconfigure defenders from $D$ to a dominating set that contains a left coset of $\langle K_k',b\rangle$ in $\tau^{1-i}\langle H_k,b\rangle$.
\end{cor}

\begin{proof}
Without loss of generality using vertex-transitivity, we may assume that $D$ contains $\langle K_k', b\rangle$ and possibly one additional vertex $u$.
There must be at least one of the cosets of $H_k$ in $\tau\langle H_k,b\rangle$ to which the defender from $u$ does not reconfigure, so all defenders who end in this coset must come from $\langle K_k',b\rangle$. Since there are no edges between $H_k$ and $b\tau H_k$ or between $bH_k$ and $\tau H_k$, 
these defenders must all in fact come from a single left coset of $K_k'$ (either $K_k'$ or $bK_k'$). Since this means only $k$ defenders end in this coset of $H_k$, the new dominating set only has $k$ vertices in this coset of $H_k$. Hence, these vertices must be a left coset of $K_k'$. This contradicts~\Cref{moving}.
\end{proof}

\begin{thm}\label{cant-defend}
Suppose $k \ge 17$ is odd, and $D$ is a dominating set of cardinality at most $2k+1$ for $\Delta_k$. Then there is a vertex $v \notin D$ such that defenders cannot reconfigure from $D$ to any dominating set that contains $v$. Therefore, $\etdom(\Delta_k)\ge 2k+2=\gamma(\Delta_k)+2$.
\end{thm}

\begin{proof}
By~\Cref{K'b-cosets}, $D$ must contain a left coset of $\langle K_k',b\rangle$.
    Without loss of generality, using vertex-transitivity, we may assume that $D$ contains $\langle K_k',b\rangle$ and possibly one additional vertex $u$. Observe that this means that $D$ contains exactly $2$ vertices in each left coset of $\langle K_k,b\rangle$ in $\langle H_k,b\rangle$, except that it may possibly contain $3$ vertices in one of these left cosets (if $u$ is one of these vertices).

    We consider $3$ possibilities: first, suppose there is no $u$. Then for any $v \in \tau\langle H_k,b \rangle$, we must reconfigure a defender from $\langle K_k',b \rangle$ to this vertex, but by~\Cref{cant-move} and the characterisation of~\Cref{K'b-cosets}, it is not possible for the reconfigured defenders to end on a dominating set.

    Next, suppose $u \in \langle H_k,b\rangle$, so there is a left coset of $\langle K_k,b\rangle$ in $\langle H_k,b\rangle$ that has $3$ defenders. Let $v$ be any vertex of $\tau\langle H_k,b\rangle$ that is not dominated by any of the three vertices occupied by these defenders (there are at least $2k^2-3(k+1)>0$ such vertices, so this is possible). Then some defender from another left coset of $\langle K_k,b\rangle$ must move to $v$, reducing the number of defenders in this left coset of $\langle K_k,b\rangle$ to $1$. However, by~\Cref{K'b-cosets} together with~\Cref{cant-move}, we must end with at least two defenders in each left coset of $\langle K_k,b\rangle$ in $\langle H_k,b\rangle$, and within $\langle H_k,b\rangle$, any neighbouring vertices lie in the same left coset of $\langle K_k,b\rangle$, so we cannot fill the deficit created by this move. 

    Finally, suppose $u \in \tau\langle H_k,b\rangle$. Using vertex-transitivity, we may assume without loss of generality that $u=\rho^i\tau a^i b$. Consider $v=\rho^{i+4}\tau a^{-i-4}$. Clearly $v$ is not a neighbour of $u$, and it is not hard to check that there is only one vertex of $D$ that dominates $v$, namely $\rho^{i+3}a^{-i-3}$, so the defender from this vertex must move to $v$. By~\Cref{K'b-cosets} and~\Cref{cant-move}, we must end with defenders on every vertex of some coset of $\langle K_k',b\rangle$ in $\langle H_k,b\rangle$. Furthermore, by~\Cref{delta-cover-diags} we must end with a defender in each left coset of $K_k$ in $\langle H_k,b\rangle$. Since the coset that contains $\rho^{i+3}a^{-i-3}$ has been vacated it must be filled, and since all adjacencies within $\langle H_k,b\rangle$ lie within left cosets of $\langle K_k,b\rangle$, and $u$ is not adjacent to any vertex of $H_k$, the defender from $\rho^{i+3}a^{-i-3}b$ must move to either $\rho^{i+4}a^{-i-2}$ or $\rho^{i+2}a^{-i-4}$. Therefore the coset of $\langle K_k',b\rangle$ that we move to must be either $\rho^{2}\langle K_k',b\rangle$, or $\rho^{-2}\langle K_k',b\rangle$. 

    Notice that no defender who remains on a vertex of $\langle K_k',b\rangle$ can move to fill the newly-vacant left coset of $K_k$ in $bH_k$. The only possibility is that this vacancy is filled by the defender from $u$. The vertex of $\rho^{2}\langle K_k',b\rangle$ that lies in the vacated left coset is $\rho^{i+4}a^{-i-2}b$
    and the vertex of $\rho^{-2}\langle K_k',b\rangle$ that lies in the vacated left coset is $\rho^{i+2}a^{-i-4}b$. But the neighbours of $u$ in $bH_k$ are all in $\rho^i b A_k \cup \rho^{i-1}bA_k$, so using the defender from $u$ is not possible.
\end{proof}

\end{document}